\documentclass{amsart}

\newcommand{\w}{\omega}
\newcommand{\IN}{\mathbb N}
\newcommand{\IZ}{\mathbb Z}
\newcommand{\IR}{\mathbb R}
\newcommand{\asdim}{\mathrm{asdim}}
\newcommand{\diam}{\mathrm{diam}}
\newcommand{\St}{\mathcal St}
\newcommand{\T}{\mathcal T}
\newcommand{\macrodim}{\mathrm{asdim}}
\newcommand{\I}{\mathcal I}

\newcommand{\E}{\mathcal E}

\newcommand{\F}{\mathcal F}
\newcommand{\mesh}{\mathrm{mesh}}
\newcommand{\U}{\mathcal U}

\newcommand{\e}{\varepsilon}

\newcommand{\Ra}{\Rightarrow}

\newcommand{\Tau}{\mathcal T}
\newcommand{\conv}{\mathrm{conv}}
\newcommand{\Lip}{\mathrm{Lip}}

\newtheorem{theorem}{Theorem}[section]
\newtheorem{proposition}[theorem]{Proposition}
\newtheorem{corollary}[theorem]{Corollary}
\newtheorem{problem}[theorem]{Problem}

\newtheorem{claim}[theorem]{Claim}
\theoremstyle{definition}

\newtheorem{definition}[theorem]{Definition}
\newtheorem{remark}[theorem]{Remark}

\title[Asymptotic dimension and small subsets in locally compact groups]{Asymptotic dimension and small subsets\\ in locally compact topological groups}
\author{Taras Banakh, Ostap Chervak, Nadya Lyaskovska}
\address{Jan Kochanowski University, Kielce, Poland, and Ivan Franko National University of Lviv, Ukraine}
\email{t.o.banakh@gmail.com, oschervak@gmail.com, lyaskovska@yahoo.com}
\subjclass{20F65; 20F69; 54F45; 55M10}
\keywords{Asymptotic dimension, locally compact group, coarse structure, small set}

\begin{document}
\begin{abstract} We prove that for a coarse space $X$ the ideal $\mathcal S(X)$ of small subsets of $X$ coincides with the ideal $\mathcal D_<(X)=\{A\subset X:\asdim(A)<\asdim(X)\}$ provided that $X$ is coarsely equivalent to an Euclidean space $\IR^n$. Also we prove that for a locally compact Abelian group $X$, the equality $\mathcal S(X)=\mathcal D_<(X)$ holds if and only if the group $X$ is compactly generated.
\end{abstract}
\maketitle

\section{Introduction}

In this paper we study the interplay between the ideal $\mathcal S(X)$
of small subsets of a coarse space $X$ and the ideal $\mathcal D_{<}(X)$ of subsets of asymptotic dimension less than $\asdim(X)$ in $X$.
We show that these two ideals coincide in spaces that are coarsely equivalent to $\IR^n$, in particular, they coincide in each compactly generated locally compact abelian group.

Let us recall that a {\em coarse space} is a pair $(X,\mathcal E)$ consisting of a set $X$ and a coarse structure $\mathcal E$ on $X$, which is a family of subsets of $X\times X$ (called {\em entourages}) satisfying the following axioms:
\begin{itemize}
\item[(A)] each $\varepsilon \in\mathcal E$ contains the diagonal $\Delta_X=\{(x,y)\in X^2:x=y\}$ and is {\em symmetric} in the sense that $\e=\e^{-1}$ where $\e^{-1}=\{(y,x):(x,y)\in\e\}$;
\item[(B)] for any entourages $\e,\delta\in\mathcal E$ there is an entourage $\eta\in\E$ that contains the composition\newline  $\delta\circ\e=\{(x,z)\in X^2:\exists y\in X$ with $(x,y)\in\e$ and $(y,z)\in\delta\}$;
\item[(C)] a subset $\delta\subset X^2$ belongs to $\E$ if $\Delta_X\subset\delta=\delta^{-1}\subset\e$ for some $\e\in\E$.
\end{itemize}
A subfamily $\mathcal B\subset\E$ is called a {\em base} of the coarse structure $\E$ if $$\E=\{\e\subset X^2:\exists \delta\in\mathcal B\mbox{ with }\Delta_X\subset\e=\e^{-1}\subset\delta\}.$$ A family $\mathcal B$ of subsets of $X^2$ is a base of a (unique) coarse structure if and only if it satisfies the axioms (A),(B).

Each subset $A$ of a coarse space $(X,\E)$ carries the induced coarse structure $\E_A=\{\e\cap A^2:\e\in\E\}$. Endowed with this structure, the space $(A,\E_A)$ is called a {\em subspace} of $(X,\E)$.

For an entourage $\e\subset X^2$, a point $x\in X$, and a subset $A\subset X$ let $B(x,\e)=\{y\in X:(x,y)\in\e\}$ be the {\em $\e$-ball} centered at $x$, $B(A,\e)=\bigcup_{a\in A}B(a,\e)$ be the $\e$-neighborhood of $A$ in $X$, and $\diam(A)=A\times A$ be the {\em diameter} of $A$. For a family $\U$ of subsets of $X$ we put $\mesh(\U)=\bigcup_{U\in\U}\diam(U)$.
\smallskip

Now we consider two basic examples of coarse spaces. The first of them is any metric space $(X,d)$ carrying
 the {\em metric coarse structure} whose base consists of the entourages $\{(x,y)\in X^2:d(x,y)<\e\}$ where $0\le\e<\infty$. A coarse space is {\em metrizable} if its coarse structure is generated by some metric.

The second basic example is a topological group $G$ endowed with the {\em left coarse structure} whose base consists of the entourages $\{(x,y)\in G^2:x\in yK\}$ where $K=K^{-1}$ runs over compact symmetric subsets of $G$ that contain the identity element $1_G$ of $G$. Let us observe that the left coarse structure on $G$ coincides with the metric coarse structure generated by any left-invariant continuous metric $d$ on $G$ which is {\em proper} in the sense that each closed ball $B(e,R)=\{x\in G:d(x,e)\le R\}$ is compact. In particular, the coarse structure on $\IR^n$, generated by the Euclidean metric coincides with the left coarse structure of the Abelian topological group $\IR^n$.

Now we recall the definitions of large and small sets in coarse spaces. Such sets were introduced in \cite{BM} and studied in \cite[\S11]{PB}. A subset $A$ of a coarse space $(X,\E)$ is called
\begin{itemize}
\item {\em large} if $B(A,\e)=X$ for some $\e\in\E$;
\item {\em small} if for each large set $L\subset X$ the set $L\setminus A$ remains large in $X$.
\end{itemize}

It follows that the family $\mathcal S(X)$ of small subsets of a coarse space $(X,\E)$ is an  ideal. A subfamily $\I\subset\mathcal P(X)$ of the power-set of a set $X$ is called an {\em ideal} if $\I$ is {\em additive} (in the sense that $A\cup B\in\I$ for all $A,B\in\I$) and
{\em downwards closed} (which means that $A\cap B\in\I$ for all $A\in\I$ and $B\subset X$).

Small sets can be considered as coarse counterparts of nowhere dense subsets in topological spaces, see \cite{BL}.
It is well-known \cite[7.4.18]{En} that the ideal of nowhere dense subsets in an Euclidean space $\IR^n$ coincides with the ideal generated by closed subsets of topological dimension $<n$ in $\IR^n$. The aim of this paper is to prove a coarse counterpart of this fundamental fact.

For this we need to recall \cite[9.4]{Roe} the definition of the asymptotic dimension $\asdim(X)$ of a coarse space $X$.

\begin{definition}\label{d1.1} The {\em asymptotic dimension} $\asdim(X)$ of a coarse space $(X,\E)$ is the smallest number $n\in\w$ such that for each entourage $\e\in\E$ there is a cover $\U$ of $X$ such that $\mesh(\U)\subset\delta$ for some $\delta\in\E$ and each $\e$-ball $B(x,\e)$, $x\in X$, meets at most $n+1$ sets $U\in\U$. If such a number $n\in\w$ does not exist, then we put $\asdim(X)=\infty$.
\end{definition}

 In Theorem~\ref{t2.7} we shall prove that $$\asdim(A\cup B)\le\max\{\asdim(A),\asdim(B)\}$$ for any subspaces $A,B$ of a coarse space $X$. This implies that for every number $n\in\w\cup\{\infty\}$ the family $\{A\subset X:\asdim(A)<n\}$ is an ideal in $\mathcal P(X)$. In particular, the family
$$\mathcal D_<(X)=\{A\subset X:\asdim(A)<\asdim(X)\}$$ is an ideal in $\mathcal P(X)$. According to \cite[9.8.4]{BS}, $\asdim(\IR^n)=n$ for every $n\in\w$.

The main result of this paper is:

\begin{theorem}\label{main} For every $n\in\IN$ the ideal $\mathcal S(X)$ of small subsets in the space $X=\IR^n$ coincides with the ideal $\mathcal D_{<}(X)$.
\end{theorem}

Theorem~\ref{main} will be proved in Section~\ref{s:main} with help of some tools of combinatorial topology.
In light of this theorem the following problem arises naturally:

\begin{problem}\label{pr1.3} Detect coarse spaces $X$ for which $\mathcal S(X)=\mathcal D_<(X)$.
\end{problem}

It should be mentioned that the class of coarse spaces $X$ with $\mathcal S(X)=\mathcal D_<(X)$ is closed under coarse equivalences.

A function $f:X\to Y$ between two coarse spaces $(X,\E_X)$ and $(Y,\E_Y)$ is called
\begin{itemize}
\item {\em coarse} if for each $\delta_X\in\E_X$ there is $\e_Y\in\E_Y$ such that for any pair $(x,y)\in\delta_X$ we get $(f(x),f(y))\in\e_Y$;
\item a {\em coarse equivalence} if $f$ is coarse and there is a coarse map $g:Y\to X$ such that $\{(x,g\circ f(x)):x\in X\}\subset \e_X$ and $\{(y,f\circ g(y)):y\in Y\}\subset\e_Y$ for some entourages $\e_X\in\E_X$ and $\e_Y\in\E_Y$.
\end{itemize}

Two coarse spaces $X,Y$ are called {\em coarsely equivalent} if there is a coarse equivalence $f:X\to Y$.

\begin{proposition}\label{coarse} Assume that coarse spaces $X,Y$ be coarsely equivalent. Then
\begin{enumerate}
\item $\asdim(X)=\asdim(Y)$;
\item $\mathcal D_<(X)=\mathcal S(X)$ if and only if $\mathcal D_{<}(Y)=\mathcal S(Y)$.
\end{enumerate}
\end{proposition}

This proposition will be proved in Section~\ref{s1.4}. Combined with Theorem~\ref{main} it implies:

\begin{corollary}\label{c1.5} If a coarse space $X$ is coarsely equivalent to an Euclidean space $\IR^n$, then $\mathcal D_{<}(X)=\mathcal S(X)$.
\end{corollary}

Problem~\ref{pr1.3} can be completely resolved for locally compact Abelian topological groups $G$, endowed with their left coarse structure. First we establish the following general fact:

\begin{theorem}\label{t1.6a} For each topological group $X$ endowed with its left coarse structure we get $\mathcal D_<(X)\subset \mathcal S(X)$.
\end{theorem}

We recall that a topological group $G$ is {\em compactly generated} if $G$ is algebraically generated by some compact subset $K\subset G$.

\begin{theorem}\label{t1.7} For an Abelian locally compact topological group $X$ the following conditions are equivalent:
\begin{enumerate}
\item $\mathcal S(X)=\mathcal D_<(X)$;
\item $X$ is compactly generated;
\item $X$ is coarsely equivalent to the Euclidean space $\IR^n$ for some $n\in\w$.
\end{enumerate}
\end{theorem}

This theorem will be proved in Section~\ref{s:t1.7}.

\begin{remark} Theorem~\ref{t1.7} is not true for non-abelian groups. The simplest counterexample is the discrete free group $F_2$ with two generators. Any infinite cyclic subgroup $Z\subset F_2$ has infinite index in $F_2$ and hence is small, yet $\asdim(Z)=\asdim(F_2)=1$.

A less trivial example is the wreath product $A\wr \IZ$ of a non-trivial finite abelian group $A$ and $\IZ$. The group $A\wr\IZ$ has asymptotic dimension 1 (see \cite{Gen}) and the subgroup $\IZ$ is small in $A\wr \IZ$ and has $\asdim(\IZ)=1=\asdim(A\wr\IZ)$. Let us recall that the group $A\wr\IZ$ consists of ordered pairs $((a_i)_{i\in\IZ},n)\in \big(\oplus^{\IZ}A)\times\IZ$ and the group operation on $A\wr\IZ$
is defined by $$((a_i),n)\ast ((b_i),m)=((a_{i+m}+b_i), n+m).$$
The group $A\wr\IZ$ is finitely-generated and meta-abelian but is not finitely presented, see \cite{Baum}.
\end{remark}

\begin{problem} Is $\mathcal S(X)=\mathcal D_<(X)$ for each connected Lie group $X$? For each discrete polycyclic group $X$?
\end{problem}

\section{The asymptotic dimension of coarse spaces}

In this section we present various characterizations of the asymptotic dimension of coarse spaces. First we fix some notation. Let $(X,\E)$ be a coarse space, $\e\in\E$ and $A\subset X$. We shall say that $A$ has diameter less than $\e$ if $\diam(A)\subset\e$ where $\diam(A)=A\times A$.
 A sequence $x_0,\dots,x_m\in X$ is called an {\em  $\e$-chain} if  $(x_i,x_{i+1})\in\e$ for all $i<m$. In this case the finite set $C=\{x_0,\dots,x_m\}$ also will be called an {\em $\e$-chain}. A set $C\subset X$ is called $\e$-{\em connected} if any two points $x,y\in C$ can be linked by an $\e$-chain $x=x_0,\dots,x_m=y$.
The maximal $\e$-connected subset $C(x,\e)\subset X$ containing a given point $x\in X$ is called the {\em $\e$-connected component} of $x$. It consists of all points $y\in X$ that can be linked with $x$ by an $\e$-chain $x=x_0,\dots,x_m=y$.

A family $\U$ of subsets of $X$ is called {\em $\e$-disjoint} if $(U\times V)\cap\e=\emptyset$ for any distinct sets $U,V\in\U$. Each natural number $n$ is identified with the set $\{0,\dots,n-1\}$.

We shall study the interplay between the asymptotic dimension introduced in Definition~\ref{d1.1} and the following its modification:

\begin{definition}\label{d2.1} The {\em colored asymptotic dimension} $\asdim_{col}(X)$ of a coarse space $(X,\E)$ is the smallest number $n\in\w$ such that for every entourage $\varepsilon\in \E$ there is a cover $\U$ of $X$ such that $\mesh(\U)\subset\delta$ for some $\delta\in\E$ and $\U$ can be written as the union $\U=\bigcup_{i\in n+1}\U_i$ of $n+1$ many $\e$-disjoint subfamilies $\U_i$. If such a number $n\in\w$ does not exist, then we put $\asdim_{col}(X)=\infty$.
\end{definition}

Without lost of generality  we can assume that the cover $\U=\bigcup_{i\in n+1}\U_i$ in
the above definition consists of pairwise disjoint sets. In this case we can consider the coloring $\chi: X\to n+1=\{0,\dots,n\}$ such that $\chi^{-1}(i)=\bigcup\U_i$ for every $i\in n+1$. For this coloring every $\chi$-monochrome $\e$-connected subset $C\subset X$ lies in some $U\in\U$ and hence has diameter $\diam(C)\subset\diam(U)\subset\mesh(\U)\subset\delta$. A subset $A\subset X$ is {\em $\chi$-monochrome} if $\chi(A)$
is a singleton. Thus we arrive to the following useful characterization of the colored asymptotic dimension.

\begin{proposition}\label{p2.2} A coarse space $(X,\E)$ has $\asdim_{col}(X)\leq n$ for some number $n\in\w$
if and only if for any $\varepsilon \in \E$ there is a coloring $\chi:
X\to n+1$ and an entourage $\delta\in \E$ such that
 each $\chi$-monochrome $\varepsilon$-chain $C\subset X$ has $\diam(C)\subset\delta$.
\end{proposition}

\begin{proof} The ``only if'' part follows from the above discussion. To prove the ``if'' part, for every $\e\in\E$ we need to construct a cover $\U=\bigcup_{i\in n+1}\U_i$ such that $\mesh(\U)\in\E$ and each family $\U_i$ is $\e$-disjoint. By our assumption, there is a coloring  $\chi:X\to n+1$ and an entourage $\delta\in \E$ such that each $\chi$-monochrome $\e$-chain $C\subset X$ has $\diam(C)\subset\delta$.

For each $x\in X$ let $C_\chi(x,\e)$ be the set of all points $y\in X$ that can be linked with $x$ by a $\chi$-monochome $\e$-chain $x=x_0,x_1,\dots,x_m=y$. It follows that $\diam(C_\chi(x,\e))\subset \delta$. For every $i\in n+1$ consider the $\e$-disjoint family $\U_i=\{C_\chi(x,\e): x\in \chi^{-1}(i)\}$. It is clear that  $\U=\bigcup_{i\in n+1}\U_i$ is a cover with $\mesh(\U)\subset\delta\in\E$, witnessing that $\asdim_{col}(X)\le n$.
\end{proof}

Now we are ready to prove the equivalence of two definitions of asymptotic dimension. For metrizable coarse spaces this equivalence was proved in \cite[9.3.7]{BS}.

\begin{proposition}\label{p2.3} Each coarse space $(X,\E)$ has $\asdim(X)=\asdim_{col}(X)$.
\end{proposition}

\begin{proof} To prove that $\asdim(X)\le\asdim_{col}(X)$, put $n=\asdim_{col}(X)$ and take any entourage $\e\in\E$. By Definition~\ref{d2.1}, for the entourage $\e\circ\e\in\E$ we can find a cover $\U=\bigcup_{i\in n+1}\U_i$ with $\mesh(\U)\in\E$ such that each family $\U_i$ is $\e\circ\e$-disjoint. We claim that each $\e$-ball $B(x,\e)$, $x\in X$, meets at most one set of each family $\U_i$. Assuming that $B(x,\e)$ meets two distinct sets $U,V\in\U_i$, we can find points $u\in U$ and $v\in V$ with $(x,u),(x,v)\in\e$ and conclude that $(u,v)\in\e\circ\e$, which is not possible as $\U_i$ is $\e\circ\e$-disjoint. Now we see that the ball $B(x,\e)$ meets at most $n+1$ element of the cover $\U$ and hence $\asdim(X)\le n$.
\smallskip

The proof of the inequality $\asdim_{col}(X)\le \asdim(X)$ is a bit longer. If the dimension $n=\asdim(X)$ is infinite, then there is nothing to prove. So, we assume that $n\in\w$. To prove that $\asdim_{col}(X)\le n$, fix an   entourage $\e\in\E$. Let $\e^0=\Delta_X$ and $\e^{k+1}=\e^k\circ \e$ for $k\in\w$. Since   $\macrodim(X)\leq n$, for the entourage $\e^{n+1}\in\E$ we can find a cover $\U$ of $X$ such that $\delta=\mesh(\U)\in\E$ and each $\e^{n+1}$-ball $B(x,\e^{n+1})$ meets at most $n+1$ many sets $U\in\U$. For every $i\le n+1$ and $x\in X$ consider the subfamily $\U(x,\e^i)=\{U\in\U:B(x,\e^i)\cap U\ne\emptyset\}$ of $\U$. It follows that $1\le|\U(x,\e^i)|\le|\U(x,\e^{i+1})|\le n+1$ for every $0\le i\le n$. Consequently, $|\U(x,\e^i)|=i$ for some $i\le n+1$. Let $\chi(x)$ be the maximal number $k\le n$ such that $|\U(x,\e^{k+1})|=k+1$. In such a way we have defined a coloring $\chi:X\to n+1=\{0,\dots,n\}$.

To finish the proof it suffices to show that any $\chi$-monochrome
$\varepsilon$-chain $C=\{x_0,\dots,x_m\}\subset X$ has $\diam(C)\subset\delta\circ \varepsilon^{n+1}$.
Let $k=\chi(x_0)$ be the color of the chain $C$. It follows that $|\U(x_i,\e^{k+1})|=k+1$ for all $x_i\in C$. We claim that $\U(x_i,\e^{k+1})=\U(x_{i+1},\e^{k+1})$ for all $i<m$. Assuming the converse, we would get that $|\U(x_i,\e^{k+1})\cup\U(x_{i+1},\e^{k+1})|\ge k+3$ and then the family $\U(x_i,\e^{k+2})\supset \U(x_i,\e^{k+1})\cup\U(x_{i+1},\e^{k+1})$ has cardinality $|\U(x_i,\e^{k+2})|\ge k+3$, which implies that $|\U(x_i,\e^i)|=i$ for some $k+3\le i\le n+1$. But this contradicts the definition of $k=\chi(x_i)$. Hence $\U(x_i,\e^{k+1})=\U(x_{0},\e^{k+1})$ for all $i\le m$ and then $C\subset B(U,\e^{k+1})$ for every $U\in\U(x_0,\e^{k+1})$. Now we see that $\diam (C)\subset \diam(U)\circ \e^{k+1}\subset\delta\circ\e^{k+1}$.
\end{proof}

Propositions ~\ref{p2.2} and ~\ref{p2.3} imply:

\begin{corollary}\label{c2.4} A coarse space $(X,\E)$ has asymptotic dimension $\macrodim(X)\leq n$ for some $n\in\w$
if and only if for any $\varepsilon \in \E$ there are $\delta\in\E$ and a coloring $\chi:X\to n+1$ such that
 any $\chi$-monochrome $\varepsilon$-chain $C\subset X$ has $\diam(C)\subset\delta$.
\end{corollary}

This corollary can be generalized as follows (cf. \cite{CZ}).

\begin{proposition}\label{p2.5} A coarse space $(X,\E)$ has $\asdim(X)\le n$ for some $n\in\w$ if and only if for any entourage
$\varepsilon\in\E$ there is an entourage $\delta\in\E$ such that for any
finite set $F\subset X$ there is a coloring $\chi:F\to n+1$ such that  each $\chi$-monochrome $\varepsilon$-chain $C\subset F$ has $\diam(C)\subset\delta$.
\end{proposition}

\begin{proof} This proposition will follow from Corollary~\ref{c2.4} as soon as for any $\e\in\E$ we find $\delta\in\E$ and a coloring $\chi:X\to n+1$ such that each $\chi$-monchome $\e$-chain in $X$ has diameter less that $\delta$.

By our assumption, there is an entourage $\delta\in\E$ such that for every finite subset $F\subset X$ there is a coloring $\chi_F:F\to n+1$ such that each $\chi_F$-monochrome $\e$-chain in $F$ has diameter less that $\delta$.
Extend $\chi_F$ to a coloring $\tilde\chi_F:X\to n+1$.

Let $\F$ denote the family of all finite subsets of $X$, partially ordered by the inclusion relation $\subset$. The colorings $\tilde\chi_F$, $F\in\F$, can be considered as elements of the compact Hausdorff space $K=\{0,\dots,n\}^X$ endowed with the Tychonov product topology. The compactness of $K$ implies that the net $\{\tilde\chi_F\}_{F\in\F}$ has a cluster point $\chi\in K$, see \cite[3.1.23]{En}. The latter means that for each finite set $F_0\in\F$ and a neighborhood $O(\chi)\subset K$ there is a finite set $F\in\F$ such that $F\supset F_0$ and $\tilde\chi_F\in O(\chi)$.

We claim that the coloring $\chi:X\to n+1$ has the required property: each $\chi$-monochome $\e$-chain $C\subset X$ has $\diam(X)\subset\delta$. Observe that the finite set $C$ determines a neighborhood $O_C(\chi)=\{f\in K:f|C=\chi|C\}$, which contains a coloring $\tilde \chi_F$ for some finite set $F\supset C$. The choice of the coloring $\chi_F=\tilde\chi_F|F$ guarantees that the set $C\subset F$ has  $\diam(C)\subset\delta$.
\end{proof}

Proposition~\ref{p2.5} admits the following self-generalization.

\begin{theorem}\label{p2.6} A coarse space $(X,\E)$ has $\asdim(X)\le n$ for some $n\in\w$ if and only if for any entourage
$\varepsilon\in\E$ there is an entourage $\delta\in\E$ such that for any finite $\e$-connected subset $F\subset X$ there is a coloring $\chi:F\to  n+1$ such that  each $\chi$-monochrome $\varepsilon$-chain $C\subset F$ has $\diam(C)\subset\delta$.
\end{theorem}

Finally, let us prove Addition Theorem for the asymptotic dimension.
For metrizable spaces this theorem is well known; see \cite[9.13]{Roe} or \cite[9.7.1]{BS}.

\begin{theorem}\label{t2.7} For any subspaces $A,B$ of a coarse space $(X,\E)$ we get $$\asdim(A\cup B)\le\max\{\asdim(A),\asdim(B)\}.$$
\end{theorem}

\begin{proof} Only the case of finite $n=\max\{\asdim(A),\asdim(B)\}$ requires the proof. Without loss of generality the sets $A$ and $B$ are disjoint. To show that $\asdim(A\cup B)\le n$ we shall apply
Corollary~\ref{c2.4}. Fix any entourage $\e\in\E$.
Since $\asdim(A)\le n$ there are an  entourage $\delta_A\in\E$ and a coloring $\chi_A:A\to n+1$ such that each $\chi$-monochrome $\e$-chain in $A$ has diamater less that $\delta_A$. Since $\asdim(B)\le n$, for the entourage $\e_B=\e\circ\delta_A\circ\e$ there are an entourage $\delta_B\in\E$ and a coloring $\chi_B:A\to\{0,\dots,n\}$ such that each $\chi$-monochrome $\e_B$-chain in $B$ has diameter less that $\delta_B$.

The union of the colorings $\chi_A$ and $\chi_B$ yields the coloring $\chi:A\cup B\to\{0,\dots,n\}$ such that $\chi|A=\chi_A$ and $\chi|B=\chi_B$. We claim that each $\chi$-monochrome $\e$-chain $C=\{x_0,\dots,x_m\}\subset A\cup B$ has $\diam(C)\subset\delta$ where $\delta=\delta_A\circ\e\circ\delta_B\circ\e\circ\delta_A$. Without loss of generality, the points $x_0,\dots,x_m$ of the chain $C$ are pairwise distinct.

If $C\subset A$, then $C$, being a $\chi_A$-monochrome $\e$-chain in $A$ has $\diam(C)\subset\delta_A\subset\delta$ and we are done. So, we assume that $C\not\subset A$. In this case $b=|C\cap B|\ge 1$ and we can choose a strictly increasing sequence $0\le k_1<k_2<\cdots<k_b\le m$ such that $\{x_{k_1},\dots,x_{k_b}\}=C\cap B$. Then $\{x_0,\dots,x_{k_1-1}\}$, being a $\chi_A$-monochrome $\e$-chain in $A$, has diameter less that $\delta_A$. Consequently, the $\e$-chain $\{x_0,\dots,x_{k_1}\}$ has diameter less that $\delta_A\circ\e\subset\e_B$.
By the same reason the $\e$-chain $\{x_{k_b},\dots,x_m\}$ has diameter less that $\e\circ\delta_A\subset\e_B$ and for every $1\le i<b$ the $\e$-chain $\{x_{k_i},\dots,x_{k_{i+1}}\}\subset \{x_{k_i}\}\cup A\cup\{x_{k_{i+1}}\}$ has diameter less than $\e\circ\delta_A\circ\e=\e_B$. Then $\{x_{k_1},\dots,x_{k_b}\}$, being a $\chi_B$-monochrome $\e_B$-chain in $B$, has diameter less that $\delta_B$. Now we see that the $\e$-chain $C=\{x_0,\dots,x_m\}$ has $\diam(C)\subset \delta_A\circ\e\circ\delta_B\circ\e\circ\delta_A=\delta$.
\end{proof}

The characterization Theorem~\ref{p2.6} will be applied to prove the following theorem which was known \cite[2.1]{DS} in the context of countable groups.

\begin{theorem}\label{t3.1} If $G$ is a topological group endowed with its left coarse structure, then
$$\asdim(G)=\sup\,\{\asdim(H):\mbox{$H$ is a compactly generated subgroup of $G$}\}.$$
\end{theorem}

\begin{proof} Let $n=\sup\,\{\asdim(H):\mbox{$H$ is a compactly generated subgroup of $G$}\}.$
It is clear that $n\le\asdim(G)$. The reverse inequality $\asdim(G)\le n$ is trivial if $n=\infty$. So, we assume that $n<\infty$. To prove that $\asdim(G)\le n$, we shall apply Theorem~\ref{p2.6}. Let $\E$ be the left coarse structure of the topological group $G$. Given any  entourage $\e\in\E$, we
should find an entourage $\delta\in\E$ such that for each finite $\e$-connected subset $F\subset G$ there is a coloring $\chi:F\to n+1$ such that each $\chi$-monochrome $\e$-chain $C\subset F$ has $\diam(C)\subset\delta$.

By the definition of the coarse structure $\E$, for the entourage $\e\in\E$ there is a compact subset $K_\e=K_\e^{-1}\subset G$ such that $\e\subset\{(x,y)\in G^2:x\in yK_\e\}$. Let $H$ be the subgroup of $G$ generated by the compact set $K_\e$, $\E_H$ be the left coarse structure of $H$, and $\e_H=\{(x,y)\in H^2:x\in yK_\e\}\in\E_H$. Since $\asdim_{col}(H)=\asdim(H)\le n$, by Proposition~\ref{p2.2}, there is a coloring $\chi_H:H\to n+1$ and an entourage $\delta_H\in\E_H$ such that each $\chi_H$-monochrome $\e$-chain $C\subset H$ has diameter $\diam(C)\subset\delta_H$. By the definition of the coarse structure $\E_H$, there is a compact subset $K_\delta=K_\delta^{-1}\ni 1_H$ of $H$ such that $\{(x,y)\in H\times H:x\in yK_\delta\}$.

We claim that the entourage $\delta=\{(x,y)\in G\times G:x\in yK_\delta\}$ satisfies our requirements. Let $F$ be a finite $\e$-connected subset of $G$. Then for each point $x_0\in F$ we get $F\in x_0 H$ and hence $x_0^{-1}F\subset H$. So, we can define a coloring $\chi:F\to n+1$ letting $\chi(x)=\chi_H(x_0^{-1}x)$ for $x\in F$. If $C\subset F$ is a $\chi$-monochrome $\e$-chain, then $x_0^{-1}C$ is a $\chi_H$-monochrome $\e_H$-chain in $H$ and hence $\diam(x_0^{-1}C)\subset\delta_H$.
The latter means that for any points $c,c'\in C$ we get $(x_0^{-1}c,x_0^{-1}c')\in\delta_H\subset \{(x,y)\in H\times H:x\in yK_\delta\}$ and hence $x_0^{-1}c\in x_0c'K_\delta$ and $c\in c'K_\delta$, which means that $(c,c')\in\delta$ and hence $\diam(C)\subset \delta$.
\end{proof}

\section{Proof of Proposition~\ref{coarse}}\label{s1.4}

Let $f:X\to Y$ be a coarse equivalence between two coarse spaces $(X,\E_X)$ and $(Y,\E_Y)$.
Then there is a coarse map $g:Y\to X$ such that
$\{(x,g\circ f(x)):x\in X\}\subset \eta_X$ and $\{(y,f\circ g(y)):y\in Y\}\subset\eta_Y$ for some entourages $\eta_X\in\E_X$ and $\eta_Y\in\E_Y$. It follows that $B(f(X),\eta_X)=Y$ and $B(g(Y),\eta_X)=X$.
\smallskip

1. First we prove that $\asdim(X)=\asdim(Y)$. Actually, this fact is known \cite[p.129]{Roe} and we present a proof for the convenience of the reader. By the symmetry, it suffices to show that $\asdim(X)\le\asdim(Y)$. This inequality is trivial if $n=\asdim(Y)$ is infinite. So, assume that $n<\infty$. By Proposition~\ref{p2.2} and \ref{p2.3}, the inequality $\asdim(X)\le n$ will be proved  as soon as for each $\e_X\in \E_X$ we find $\delta_X\in\E_X$ and a coloring $\chi_X:X\to n+1$ such that each $\chi_X$-monochrome $\e_X$-chain $C\subset X$ has diameter $\diam(C)\subset\delta_X$.

Since the map $f:X\to Y$ is coarse, for the entourage $\e_X$ there is an entourage $\e_Y$ such that $\{(f(x),f(x')):(x,x')\in\e_X\}\subset\e_Y$. Since $\asdim(Y)=n$, for the entourage $\e_Y$ there is an entourage $\delta_Y\in\E_Y$ and a coloring $\chi_Y:Y\to n+1$ such that each $\chi_Y$-monochrome $\e_Y$-chain $C_Y\subset Y$ has diameter $\diam(C_Y)\subset\delta_Y$.

Since the function $g:Y\to X$ is coarse, for the entourage $\delta_Y$ there is an entourage $\delta_X'$ such that $\{(g(y),g(y')):(y,y')\in\delta_Y\}\subset\delta_Y'$. Put $\delta_X=\eta_X\circ\delta_X'\circ\eta_Y$ and consider the coloring $\chi_X=\chi_Y\circ f:X\to n+1$ of $X$. We claim that each $\chi_X$-monochrome $\e_X$-chain $C_X\subset X$ has diameter $\diam(C_X)\subset\delta_X$. Then choice of $\e_Y$ guarantees that the set $C_Y=f(C_X)$ is an $\e_Y$-chain. Being $\chi_X$-monochrome, it has diameter $\diam(C_Y)\subset\delta_Y$. Then the set $C_X'=g(C_Y)$ has diameter $\diam(C_X')\subset\delta_X'$. Now take any two points $c,c'\in C_X$ and observe that the pairs $(c,g\circ f(c))$ and $(c',g\circ f(c'))$ belong to the entourage $\eta_X$. Consequently, $$(c,c')\in\{(c,g\circ f(c))\}\circ\{(g\circ f(c),g\circ f(c')\}\circ\{(g\circ f(c'),c')\}\subset \eta_X\circ\delta'_X\circ\eta_X=\delta_X$$ which means that the $\e_X$-chain $C_X$ has diameter $\diam(C_X)\subset\delta_X$. So, $\asdim(X)\le n$.
\smallskip

2. The second statement of Proposition~\ref{coarse}, follows Claims~\ref{cl7a} and \ref{cl7c} proved below.

\begin{claim}\label{cl7a} A subset $A\subset X$ and its image $f(A)\subset Y$ have the same asymptotic dimension $\asdim(A)=\asdim(f(A))$.
\end{claim}

\begin{proof}
 This claim follows from Proposition~\ref{coarse}(1) proved above, since $A$ and $f(A)$ are coarsely equivalent.
\end{proof}

\begin{claim}\label{cl7b} A subset $A\subset X$ is large in $X$ if and only if its image $f(A)$ is large in $Y$.
\end{claim}

\begin{proof}
If $A$ is large in $X$, then $B(A,\e_X)=X$ for some $\e_X\in\E_X$. Since $f$ is coarse, there exists $\e_Y\in\E_Y$ such that for each $(x_0,x_1)\in\e_X$ we get $(f(x_0),f(x_1))\in \e_Y$.
It follows that $B(f(A),\e_Y)\supset f(Y)$ and $B(f(A),\e_Y\circ\eta_Y)=B(B(f(A),\e_Y),\eta_Y)\supset B(f(X),\eta_Y)=Y$, which means that $f(A)$ is large.

Now assume conversely that the set $f(A)$ is large in $Y$. Then $g\circ f(A)$ is large in $X$. Since
$g\circ f(A)\subset B(A,\eta_X)$, we conclude that $A$ in large in $X$.
So, $A$ is large in $X$ if and only if $f(A)$ is large in $Y$.
\end{proof}

\begin{claim}\label{cl7c} A subset $A\subset X$ is small if and only if for each entourage $\e_X\in\E_X$ the set $B(A,\e_x)$ is small.
\end{claim}

\begin{proof} The ``if'' part is trivial. To prove the ``only if'' part, assume that the set $A$ is small in $X$. To show that $B(A,\e_X)$ is small in $X$, it is necessary to check that for each large subset $L\subset X$ the complement $L\setminus B(A,\e_X)$ is large in $X$. Consider the set $L'=(L\setminus B(A,\e_X))\cup A$ and observe that $L\subset B(L',\e_X)$ and hence $L'$ is large in $X$. Since $A$ is small, the set $L'\setminus A=L\setminus B(A,\e_X)$ is large in $X$.
\end{proof}

\begin{claim}\label{cl7d} A subset $A\subset X$ is small in $X$ if and only if its image $f(A)$ is small in $Y$.
\end{claim}

\begin{proof} Assume that $A$ is small in $X$. To prove that $f(A)$ is small in $Y$, we need to check that for any large subset $L\subset Y$ the complement $L\setminus f(A)$ is large in $Y$. Claim~\ref{cl7b} implies that the set $g(L)$ is large in $X$. By Claim~\ref{cl7c}, the set $B(A,\eta_X)$ is small in $X$ and hence the complement $g(L)\setminus B(A,\eta_X)$ remains large in $X$. By Claim~\ref{cl7b} $f(g(L)\setminus B(A,\eta_X))$ is large in $Y$. We claim that $f(g(L)\setminus B(A,\eta_X))\subset B(L\setminus f(A),\eta_Y)$. Indeed, given point $y\in f(g(L)\setminus B(A,\eta_X))$, find a point
$x\in g(L)\setminus B(A,\eta_X)$ such that $y=f(x)$ and a point $z\in L$ such that $x=g(z)$.
We claim that $z\notin f(A)$. Assuming conversely that $z\in f(A)$, we get $x=g(z)\in g\circ f(A)\subset B(A,\eta_X)$, which contradicts the choice of $x$. So, $z\in L\setminus f(A)$ and
$y=f\circ g(z)\in B(z,\eta_Y)\subset B(L\setminus f(A),\eta_Y)$.

Taking into account that the set $f(g(L)\setminus B(A,\eta_X))\subset B(L\setminus f(A),\eta_Y)$ is large in $Y$, we conclude that the set $L\setminus f(A)$ is large in $Y$ and hence $f(A)$ is small in $Y$.

Now assume that the set $f(X)$ is small in $Y$. Then the set $g\circ f(A)$ is small in $X$ and so are  the sets $B(g\circ f(A),\eta_X)\supset A$.
\end{proof}

\section{Proof of Theorem~\ref{t1.6a}}

Let $G$ be a topological group and $\E$ be its left coarse structure. The inclusion $\mathcal D_{<}(G)\subset \mathcal S(G)$ will follow as soon as we prove that each non-small subset $A\subset G$ has asymptotic dimension $\asdim(A)=\asdim(G)$. We divide the proof of this fact into 3 steps.

\begin{claim}\label{cl1} There is an entourage $\e_A\in\E$ such that the set $G\setminus B(A,\e_A)$ is not large in $G$.
\end{claim}

\begin{proof} Since $A$ is not small, there is a large set $L\subset X$ such that the complement $L\setminus A$ is not large. Since $L$ is large in $X$, there is an entourage $\e_A\in \E$ such that $B(L,\e_A)=G$. We claim that the set $G\setminus B(A,\e_A)$ is not large. Assuming the opposite,  we can find an  entourage $\delta\in\E$ such that $B(G\setminus B(A,\e_A),\delta)=G$. Then for each $x\in G$ the ball $B(x,\delta)$ meets $G\setminus B(A,\e_A)$ at some point $y$. By the choice of $\e_A$, the ball $B(y,\e_A)$ meets the large set $L$ at some point $z$. It follows from $y\notin B(A,\e_A)$ that $z\in L\cap B(y,\e)\subset L\cap(X\setminus A)=L\setminus A$ and hence $x\in B(L\setminus A,\e_A\circ\delta)$, which means that $L\setminus A$ is large in $X$. This is a required contradiction.
\end{proof}

\begin{claim} $\asdim(B(A,\e_A))=\asdim(A)$.
\end{claim}

\begin{proof} Observe that the identity embedding $i:A\to B(A,\e_A)$ is a coarse equivalence. The coarse inverse $j:B(A,\e_A)\to A$ to $i$ can be defined by choosing a point $j(x)\in B(x,\e_A)\cap A$ for each $x\in B(A,\e_A)$. Now we equality $\asdim(B(A,\e_A))=\asdim(A)$ follows from the invariance of the asymptotic dimension under coarse equivalences, see Proposition~\ref{coarse}.
\end{proof}

\begin{claim} $\asdim(G)=\asdim(A)$.
\end{claim}

\begin{proof} The inequality $\asdim(A)\le\asdim(G)$ is trivial. So, it suffices to check that $\asdim(A)\le n$ where $n=\asdim(A)=\asdim(B(A,\e_A))$. If $n$ is infinite, then there is nothing to prove. So, we assume that $n\in\w$.

For the proof of the inequality $\asdim(G)\le n$, we shall apply Theorem~\ref{p2.6}.
Given any $\e\in\E$ we should find $\delta\in\E$ such that for each finite $\e$-connected subset $F\subset G$  there is a coloring $\chi:F\to n+1$ such that each $\chi$-monochrome $\e$-chain $C\subset F$ has $\diam(C)\subset\delta$. By the definition of the left coarse structure $\E$ we lose no generality assuming that $\e=\{(x,y)\in G\times G:x\in yK_\e\}$ for some compact subset $K_\e=K_\e^{-1}\subset G$ containing the neutral element $1_G$ of $G$. In this case the entourage $\e$ is left invariant in the sense that for each pair $(x,y)\in\e$ and each $z\in G$ the pair $(zx,zy)$ belongs to $\e$.

 Since $\asdim_{col}(B(A,\e_A))=\asdim(B(A,\e_A))\le n$, for the entourage $\e\in\E$, there are an entourage $\delta\in\E$ and a coloring $\chi_A:B(A,\e_A)\to n+1$ such that each $\chi$-monochrome $\e$-chain $C\subset B(A,\e_A)$ has $\diam(C)\subset\delta$, see Proposition~\ref{p2.2}.
By the definition of the left coarse structure $\E$, we lose no generality assuming that $\delta=\{(x,y)\in G\times G:x\in yK_\delta\}$ for some compact set $K_\delta=K_\delta^{-1}\ni 1_G$ of $G$, which implies that the entourage $\delta$ is left invariant.

Now take any finite $\e$-connected subset $F\subset G$. Replacing $F$ by $F\cup F^{-1}\cup\{1_G\}$ we can assume that $F=F^{-1}\ni 1_G$. Since the set $G\setminus B(A,\e_A)$ is not large, there is a point $z\notin (G\setminus B(A,\e_A))F$. Then $zF^{-1}$ is disjoint with $G\setminus B(A,\e_A)$ and hence $zF=zF^{-1}\subset B(A,\e_A)$. So, it is legal to define a coloring $\chi:F\to n+1$ by the formula $\chi(x)=\chi_A(zx)$ for $x\in F$. Taking into account the left invariance of the entourages $\e$ and $\delta$, it is easy to see that each $\chi$-monochrome $\e$-chain $C\subset F$ has diameter $\diam(C)\subset\delta$.
By Propositions~\ref{p2.2} and \ref{p2.3}, $\asdim(G)=\asdim_{col}(G)\le n=\asdim(A)$.
\end{proof}

\section{Proof of Theorem~\ref{main}}\label{s:main}

We need to prove that a subset $A\subset\IR^n$ is small if and only if it has asymptotic dimension $\asdim(A)<\asdim(\IR^n)=n$. The ``if'' part of this characterization follows from the inclusion $\mathcal D_{<}(\IR^n)\subset\mathcal S(\IR^n)$ proved in  Theorem~\ref{t1.6a}. To prove the ``only if'' part, we need to recall some (standard) notions of Combinatorial Topology \cite{Pont}, \cite{Miod}.

On the Euclidean space $\IR^n$ we shall consider the metric generated by the sup-norm
$\|x\|=\max_{i\in n}|x(i)|$.

By the {\em standard $n$-dimensional simplex} we understand the compact convex subset
$$\Delta=\big\{(x_0,\dots,x_n)\in[0,1]^{n+1}:\sum_{i=0}^nx_i=1\big\}\subset\IR^{n+1}$$
of the Euclidean space $\IR^{n+1}$ endowed with the sup-norm. For each $i\le n$ by $v_i:n+1\to\{0,1\}\subset \IR$ we denote the vertex of $\Delta$ defined by $v_i^{-1}(1)=\{i\}$. For each vertex $v_i$ of $\Delta$ consider its
{\em star}
$$\St_\Delta(v_i)=\{x\in \Delta:x(i)>0\}$$and its
{\em barycentric star}
$$\St'_\Delta(v_i)=\big\{x\in\Delta:x(i)=\max_{j\le n}x(j)\big\}\subset\St_\Delta(v_i).$$
It is clear that $\bigcup_{i=0}^n \St'_\Delta(v_i)=\Delta$ while $\bigcap_{i=0}^n\St'_\Delta(v_i)=\{b_\Delta\}$ is the singleton containing the barycenter $$b_\Delta=\tfrac1{n+1}\sum_{i=0}^nv_i$$ of the simplex $\Delta$.

\begin{claim}\label{cl4} $\bigcap_{i=0}^n B(St'_\Delta(v_i),\e)\subset B(b_\Delta,n\e)$ for each positive real number $\e$.
\end{claim}

\begin{proof} Given any vector $x\in \bigcap_{i=0}^n B(St'_\Delta(v_i),\e)$, for every $i\le n$ we can find a vector $y\in \St'_\Delta(v_i)$ with $\|x-y\|<\e$. Then $|x_i-y_i|\le\|x-y\|<\e$ and hence $x_i>y_i-\e=\max_{j\le n}y_j-\e\ge \frac1{n+1}-\e$. On the other hand, $$x_i=1-\sum_{j\ne i}x_j<1-\sum_{j\ne i}\big(\frac1{n+1}-\e\big)=1-\frac{n}{n+1}+n\e=\frac1{n+1}+n\e.$$ So, $\|x-b_\Delta\|<n\e$.
\end{proof}

Now we are going to generalize Claim~\ref{cl4} to arbitrary simplexes.
By an {\em n-dimensional simplex} in $\IR^n$ we understand the convex hull $\sigma=\conv(\sigma^{(0)})$ of an affinely independent subset $\sigma^{(0)}\subset\IR^n$ of cardinality $|\sigma^{(0)}|=n+1$. Each point $v\in\sigma^{(0)}$ is called a {\em vertex} of the simplex $\sigma$. The arithmetic mean
$$b_\sigma=\frac1{n+1}\sum_{v\in\sigma^{(0)}}v$$ of the vertices is called the {\em barycenter} of the simplex $\sigma$. By $\partial \sigma$ we denote the boundary of the simplex $\sigma$ in $\IR^n$.  Observe that the homothetic copy $\frac12b_\sigma+\frac12\sigma=\{\frac12b_\sigma+\frac12x:x\in\sigma\}$ of $\sigma$ is contained in the interior $\sigma\setminus\partial\sigma$ of $\sigma$.
For each vertex $v\in\sigma^{(0)}$ let $$\St_\sigma(v)=\sigma\setminus\conv(\sigma^{(0)}{\setminus}\{v\})$$ be the {\em star} of $v$ in $\sigma$.
\smallskip

In fact, $n$-dimensional simplexes can be alternatively defined an images of the standard $n$-dimensional simplex $\Delta$ under injective affine maps $f:\Delta\to\IR^n$.

A map $f:\Delta\to \IR^n$ is called {\em affine} if $f(tx+(1-t)y)=tf(x)+(1-t)f(y)$ for any points $x,y\in\Delta$ and a real number $t\in[0,1]$. It is well-known that each affine function $f:\Delta\to\IR^n$ is uniquely defined by its restriction $f|\Delta^{(0)}$ to the set $\Delta^{(0)}=\{v_i\}_{i\le n}$ of vertices of $\Delta$.

A map $f:\Delta\to\IR^n$ will be called {\em $b_\Delta$-affine} if for every $i\le n$ the restriction $f|\conv(\{b_\Delta\}{\cup}\Delta^{(0)}{\setminus}\{v_i\})$ is affine. A $b_\Delta$-affine function $f:\Delta\to\IR^n$ is uniquely determined by its restriction $f|\Delta^{(0)}\cup\{b_\Delta\}$.

A function $f:X\to Y$ between metric spaces $(X,d_X)$ and $(Y,d_Y)$ is called {\em Lipschitz} if it its {\em Lipschitz constant} $$\Lip(f)=\sup\Big\{\frac{d_Y(f(x),f(x'))}{d_X(x,x')}:x,x'\in X,\;\;x\ne x'\Big\}$$is finite. A bijective function $f:X\to Y$ is {\em bi-Lipschitz} if $f$ and $f^{-1}$ are Lipschitz.

\begin{claim}\label{cl5} For any $n$-dimensional simplex $\sigma$ in $\IR^n$ there is a real constant $L$ such that each $b_\Delta$-affine function $f:\Delta\to\sigma$ with $f(\Delta^{(0)})=\sigma^{(0)}$ and $f(b)\in\frac12b_\sigma+\frac12\sigma$  is bijective, bi-Lipschitz and has $\Lip(f)\cdot \Lip(f^{-1})\le L$.
\end{claim}

This claim can be easily derived from the fact that each $b_\Delta$-affine function $f:\Delta\to\sigma$ with $f(\Delta^{(0)})=\sigma^{(0)}$ is Lipschitz and its Lipschitz constant $\Lip(f)$ depends continuously on $f(b_\Delta)$.

Given an $n$-dimensional simplex $\sigma\subset\IR^n$ and a point $b'\in\sigma\setminus\partial\sigma$ in its interior, fix a $b_\Delta$-affine function $f:\Delta\to\sigma$ such that $f(\Delta^{(0)})=\sigma^{(0)}$ and $f(b_\Delta)=b'$. For each vertex $v\in\sigma^{(0)}$ consider its {\em $b'$-barycentric star}  $$\St'_{\sigma,b'}(v)=f\big(\St'_\Delta(f^{-1}(v))\big)\subset\St_\sigma(v).$$ It is easy to see that the set $\St_{\sigma,b'}(v)$ does not depend on the choice of the $b_\Delta$-affine function $f$.

\begin{claim}\label{cl6} For any $n$-dimensional simplex $\sigma$ in $\IR^n$ there is a real constant $L$ such that for each point $b'\in\frac12b_\sigma+\frac12\sigma$ and each $\e>0$ we get $\sigma\cap \bigcap_{v\in\sigma^{(0)}}B(\St'_{\sigma,b'}(v),\e)\subset B(b',L\e)$.
\end{claim}

\begin{proof} By Claim~\ref{cl5}, there is a real constant $C$ such that each bijective $b_\Delta$-affine function $f:\Delta\to\sigma$ with $f(\Delta^{(0)})=\sigma^{(0)}$ and $f(b_\Delta)\in\frac12b_\sigma+\frac12\sigma$ has $\Lip(f)\cdot \Lip(f^{-1})\le C$. Put $L=nC$. Given any point $b'\in\frac12b_\sigma+\frac12\sigma$,
 choose a bijective $b_\Delta$-affine function $f:\Delta\to\sigma$ such that $f(\Delta^{(0)})=\sigma^{(0)}$ and $f(b_\sigma)=b'$. The choice of $C$ guarantees that $\Lip(f)\cdot\Lip(f^{-1})\le C$.
Now observe that
$$
\begin{aligned}
\sigma\cap\bigcap_{v\in\sigma^{(0)}}B(\St'_{\sigma,b'}(v),\e)
&=\bigcap_{v\in\sigma^{(0)}}f\circ f^{-1}\big(B(\St'_{\sigma,b'}(v),\e)\big)\subset
\bigcap_{v\in\sigma^{(0)}}f\big(B(f^{-1}(\St'_{\sigma,b'}(v)),\Lip(f^{-1})\e)\big)=\\
&=\bigcap_{v\in\sigma^{(0)}}f\big(B(\St'_{\Delta}(f^{-1}(v)),\Lip(f^{-1})\e)\big)=
f\Big(\bigcap_{v\in\Delta^{(0)}}B\big(\St'_{\Delta}(v),\Lip(f^{-1})\e\big)\Big)\subset\\
&\subset f\big(B(b_\Delta,n\Lip(f^{-1})\e)\big)\subset B\big(f(b_\Delta),\Lip(f)\Lip(f^{-1})n\e\big)=B(b',Cn\e)=B(b',L\e).
\end{aligned}
$$
\end{proof}

Now consider the binary unit cube $K=\{0,1\}^n\subset\IR^n$ endowed with the partial  ordering $\le$ defined by $x\le y$ iff $x(i)\le y(i)$ for all $i<n$. Given two vectors $x,y\in\{0,1\}^n$, we write $x<y$ if $x\le y$ and $x\ne y$.

For every increasing chain $v_0<v_1<\ldots<v_n$ of points of the binary cube $K=\{0,1\}^n$, consider the simplex $\conv\{v_0,\ldots,v\}$ and let $\T_K$ be the (finite) set of these simplexes.  Next, consider the family $\Tau=\{\sigma+z:\sigma\in\Tau_K,\;z\in \IZ^n\}$ of translations of the simplexes from the family $\Tau_K$, and observe that $\bigcup\Tau=\IR^n$. For each point $v\in\IZ^n$ let
$$\St_{\Tau}(v)=\bigcup\{\St_\sigma(v):v\in\sigma\in\Tau\}$$be the {\em $\Tau$-star} of $v$ in the triangulation $\Tau$ of the space $\IR^n$.
\smallskip

Now we are able to prove the ``only if'' part of Theorem~\ref{main}. Assume that a subset $A\subset\IR^n$ is small. Then there is a function $\varphi:(0,\infty)\to (0,\infty)$ such that for each $\delta\in(0,\infty)$ and a point $x\in\IR^n$ there is a point $y\in\IR^n$ with $B(y,\delta)\subset B(x,\varphi(\delta))\setminus A$. The inequality $\asdim(A)<n$ will follow as soon as given any $\delta<\infty$ we construct a cover $\U$ of $A$ with finite $\mesh(\U)=\sup_{U\in\U}\diam(U)$ such that each $\delta$-ball $B(a,\delta)$, $a\in A$, meets at most $n$ elements of the cover $\U$.

By Claim~\ref{cl6}, there is a constant $L$ such that for each simplex $\sigma\in \Tau$, each point $b'\in \frac12 b_\sigma+\frac12\sigma$ and each $\e>0$ we get $\sigma\cap\bigcap\limits_{v\in\sigma^{(0)}}B(\St'_{\sigma,b'}(v),\e)\subset B(b',L\e)$.

Given any $\delta<\infty$, choose $\e>0$ so small that
for any simplex $\sigma\in\Tau$ the following conditions hold:
\begin{enumerate}
\item $B\big(b_\sigma,\e\varphi(L\delta)\big)\subset \frac12b_\sigma+\frac12\sigma$;
\item for any $b'\in\frac12b_\sigma+\frac12\sigma$ and any vertex $v\in\sigma^{(0)}$ the $2\e\delta$-neighborhood $B(\St'_{\sigma,b'}(v),2\e\delta)$ lies in the $\Tau$-star $\St_\Tau(v)$ of $v$.
\end{enumerate}

Now consider the closed cover $$\widetilde\Tau=\{\e^{-1}\sigma:\sigma\in\Tau\}$$of the space $\IR^n$ and observe that for each simplex $\sigma\in\widetilde\Tau$ we get
\begin{itemize}
\item[$(1_\e)$] $B\big(b_\sigma,\varphi(L\delta)\big)\subset \frac12b_\sigma+\frac12\sigma$;
\item[$(2_\e)$]  for any $b'\in\frac12b_\sigma+\frac12\sigma$ and any vertex $v\in\sigma^{(0)}$ the $2\delta$-neighborhood $B(\St'_{\sigma,b'}(v),2\delta)$ lies in the $\widetilde\Tau$-star $\St_{\widetilde\Tau}(v)$ of $v$.
\end{itemize}

By the choice of the function $\varphi$, for each simplex $\sigma\in\widetilde\Tau$, there is a point $b_{\sigma}'\in\IR^n$ such that $B(b_{\sigma}',L\delta)\subset B(b_\sigma,\varphi(L\delta))\setminus A$. The condition $(1_\e)$ guarantees that $$b_{\sigma}'\in B(b_\sigma,\varphi(L\delta)) \subset \tfrac12b_\sigma+\tfrac12\sigma.$$

For every point $v\in \frac1\e\IZ^n$ consider the set
$$\St'(v)=\bigcup\{\St_{\sigma,b'_\sigma}(v):\sigma\in\widetilde\Tau,\;v\in\sigma^{(0)}\}\subset \St'_{\widetilde\Tau}(v)$$ and observe that $\U=\{\St'(v):v\in{\e^{-1}}\IZ^n\}$ is a cover of the Euclidean space $\IR^n$. It follows that
$$\mesh(\U)=\sup_{v\in\e^{-1}\IZ^n}\diam(\St'(v))\le 2\sup_{v\in\e^{-1}\IZ^n}\diam(\sigma)\le 2\e^{-1}\diam([0,1]^n)<\infty.$$
It remains to check that each ball $B(a,\delta)$, $a\in A$, meets at most $n$ sets $U\in\U$.

Assume conversely that there are a point $a\in A$ and a set $V\subset\e^{-1}\IZ^n$ of cardinality $|V|=n+1$ such that $B(a,\delta)\cap \St'(v)\ne\emptyset$ for each $v\in V$. Then $a\in\bigcap_{v\in V}B(\St'(v),\delta)$. It follows from $a\in \bigcap_{v\in V}B(\St'(v),\delta)\subset\bigcap_{v\in V}\St_{\widetilde\Tau}(v)$ that $V$ coincides with the set $\sigma^{(0)}$ of vertices of some simplex $\sigma\in\widetilde\Tau$ and $a$ lies in the interior of the simplex $\sigma$.

Next, we show that $a\in B(\St'_{\sigma,b_\sigma'}(v),\delta)$ for each $v\in V$. In the opposite case, $a\in B(\St'_{\tau,b_\tau'}(v),\delta)\subset B(\tau,\delta)$ for some simplex $\tau\in\widetilde\Tau\setminus\{\sigma\}$ such that $v\in\tau^{(0)}\setminus\sigma^{(0)}$.
Choose a vertex $u\in\sigma^{(0)}\setminus\tau^{(0)}$ and observe that the condition $(2_\e)$ implies that $a\in B(\St'_{\sigma,b_\sigma'},\delta)\cap B(\tau,\delta)=\emptyset$,
which is a contradiction.

Finally, the choice of $L$ and $b_\sigma'$ yields  the desired contradiction $$a\in\sigma\cap \bigcap_{v\in\sigma^{(0)}}B(\St'_{\sigma,b_\sigma'}(v),\delta)\subset B(b_\sigma',L\delta)\subset \IR^n\setminus A,$$ completing the proof of the theorem.

\section{Proof of Theorem~\ref{t1.7}}\label{s:t1.7}

Given an Abelian locally compact topological group $G$ endowed with its left coarse structure, we need to prove the equivalence of the following statements:
\begin{enumerate}
\item $\mathcal S(G)=\mathcal D_<(G)$;
\item $G$ is compactly generated;
\item $G$ is coarsely equivalent to an Euclidean space $\IR^n$ for some $n\in\w$.
\end{enumerate}

We shall prove the implications $(1)\Ra(2)\Ra(3)\Ra(1)$.
The implication $(3)\Ra(1)$ follows from Corollary~\ref{c1.5}.

To prove that $(2)\Ra(3)$, assume that the group $G$ is compactly generated. By Theorem 24 \cite[p.85]{Mor}, $G$ is topologically isomorphic to the direct sum $\IR^n\times\IZ^m\times K$ for some $n,m\in\w$ and a compact subgroup $K\subset G$. Since the projection $\IR^n\times\IZ^m\times K\to \IR^n\times\IZ^m$ and the embedding $\IZ^n\times\IZ^m\to \IR^n\times\IZ^m$ are coarse equivalences, we conclude that $G$ is coarsely equivalent to $\IZ^{n+m}$ and to $\IR^{n+m}$.

To prove that $(1)\Ra(2)$, assume that $\mathcal S(G)=\mathcal D_{<}(G)$.
First we prove that $G$ has finite asymptotic dimension. By the Principal Structure Theorem 25 \cite[p.26]{Mor}, $G$ contains an open subgroup $G_0$ that is topologically isomorphic to $\IR^n\times K$ for some $n\in\w$ and some compact subgroup $K$ of $G_0$. The subgroup $G_0$ has asymptotic dimension $\asdim(G_0)=\asdim(\IR^n)=n<\infty$. If $\asdim(G)=\infty$, then the quotient group $G/G_0$ has infinite asymptotic dimension and hence has infinite free rank. Then the group $G/G_0$ contains a subgroup isomorphic to the free abelian group $\oplus^\w\IZ$ with countably many generators. It follows that $G$ also contains a discrete subgroup $H$ isomorphic to $\oplus^\w\IZ$.
Replacing $H$ by a smaller subgroup, if necessary, we can assume that $H$ has infinite index in $G$ and hence is small in $G$. Since $\asdim(H)=\infty=\asdim(G)$, we conclude that $\mathcal S(G)\ne\mathcal D_<(G)$, which is a desired contradiction showing that $\asdim(G)<\infty$.

By Theorem~\ref{t3.1}, there is a compactly generated subgroup $H\subset G$ with $\asdim(H)=\asdim(G)$. Since $H\notin \mathcal D_{<}(G)=\mathcal S(G)$, the subset $H$ is not small in $G$. Repeating the proof of Claim~\ref{cl1}, we can show that the set $G\setminus B(H,\e)$ is not large for some entourage $\e\in\E$. By the definition of the left coarse structure $\E$, there is a compact subset $K\subset G$ such that $B(H,\e)\subset HK$. We claim that $K^{-1}HK=G$. Assuming the opposite, we can find a point $x\in G\setminus K^{-1}HK$ and consider the finite set $F=\{x,x^{-1},xx^{-1}\}=F^{-1}$. Since the set $G\setminus HK$ is not large, there is a point $z\in (G\setminus HK)F$. For this point $z$ we get $zF\cap (G\setminus HK)=\emptyset$ and hence $z\in zF\subset HK$. Then $x\in z^{-1}zF\subset z^{-1}HK\subset K^{-1}HHK=K^{-1}HK$, which is a contradiction. Now the compact generacy of the subgroup $H$ implies the compact generacy of the group $G=K^{-1}HK$.

\end{document}